\documentclass{amsart}[11pt]

\def\volume{\operatorname{vol}}

\def\op{\operatorname}
\def\svolball#1#2{{\volume(\underline B_{#2}^{#1})}}
\def\svolann#1#2{{\volume(\underline A_{#2}^{#1})}}

\def\svolsp#1#2{{\volume(\partial \underline B_{#2}^{#1})}}

\allowdisplaybreaks
\usepackage[usenames,dvipsnames]{color}
\usepackage{amsmath,amssymb,bm,amsthm, amsfonts, mathrsfs, eucal, epsfig}
\usepackage{graphicx}
\usepackage{url}
\usepackage[top=2.5cm,bottom=2.0cm,left=2.5cm,right=2.5cm]{geometry}

\makeatletter 
\@addtoreset{equation}{section}
\makeatother  

\begin{document}

\newtheorem{Thm}{Theorem}[section]
\newtheorem{Def}{Definition}[section]
\newtheorem{Lem}[Thm]{Lemma}
\newtheorem{Rem}[Thm]{Remark}
\newtheorem{Cor}[Thm]{Corollary}
\newtheorem{sublemma}{Sub-Lemma}
\newtheorem{Prop}{Proposition}[section]
\newtheorem{Example}{Example}[section]
\newcommand{\g}[0]{\textmd{g}}
\newcommand{\pr}[0]{\partial_r}
\newcommand{\dif}{\mathrm{d}}
\newcommand{\bg}{\bar{\gamma}}
\newcommand{\md}{\rm{md}}
\newcommand{\cn}{\rm{cn}}
\newcommand{\sn}{\rm{sn}}
\newcommand{\seg}{\mathrm{seg}}

\newcommand{\Ric}{\mbox{Ric}}
\newcommand{\Iso}{\mbox{Iso}}
\newcommand{\ra}{\rightarrow}
\newcommand{\Hess}{\mathrm{Hess}}
\newcommand{\RCD}{\mathsf{RCD}}
\title{Segment inequality and almost rigidity structures for integral Ricci curvature}
\author{Lina Chen}
\address[Lina Chen]{Department of mathematics, Nanjing University, Nanjing China}

\email{chenlina\_mail@163.com}
\thanks{Supported partially by NSFC Grant 12001268  and a research fund from Nanjing University.} 
\date{\today}

\maketitle

\begin{abstract}

\setlength{\parindent}{10pt} \setlength{\parskip}{1.5ex plus 0.5ex
minus 0.2ex} 

We will show the Cheeger-Colding segment inequality for manifolds with integral Ricci curvature bound. By using this segment inequality, the almost rigidity structure results for integral Ricci curvature will be derived by a similar method as in \cite{CC1}.  And the sharp H\"older continuity result of \cite{CoN} holds in the limit space of manifolds with integral Ricci curvature bound.
 \end{abstract}

\section{Introduction}

This paper is concerned with Riemannian manifolds with integral Ricci curvature bound. Consider a $n$-manifold $M$. For each $x\in M$ let $\rho\left( x\right) $ denote the smallest
eigenvalue for the Ricci tensor $\op{Ric} : T_xM\to T_xM$. For a constant $H$, let $\rho_H=\max\{-\rho(x)+(n-1)H, 0\}$ and for $R>0$, let
$$\bar k(H,p)=\left(\frac{1}{\volume(M)}\int_M \rho_H^p dv\right)^{\frac{1}{p}}= \left(-\kern-1em\int_M \rho_H^p dv\right)^{\frac{1}{p}}, \quad \bar k(H ,p, R)= \sup_{x\in M}\left(-\kern-1em\int_{B_R(x)} \rho_H^p dv\right)^{\frac{1}{p}}.$$ 
 Then  $\bar k(H, p)/\bar k(H, p, R)$  measures the average amount of Ricci
 curvature lying below a given bound, in this case,  $(n-1)H$, in the $L^p$ sense. Clearly $\bar k(H, p) = 0$ iff $\Ric_M \ge (n-1)H$.
 
For a complete $n$-manifold $M$ with Ricci curvature lower bound, $\op{Ric}_M\geq (n-1)H$, Cheeger-Colding \cite{CC1} showed that the almost rigidity properties hold on $M$:  almost volume cone implies almost metric cone and almost splitting theorem holds. And in \cite{CC2, CC3, CC4}, they studied the degeneration of the convergent sequences of manifolds with lower Ricci curvature bound and derived many fundamental properties about the regularity and singularity of the limit spaces and the stability of the sequences. For instance, they proved that the regular set has full measure and volume is convergence in the non-collapsing case (see \cite{Co} for the smooth case). Since then many importance works about manifolds with bounded Ricci curvature were done, like Colding-Naber's (\cite{CoN}) sharp H\"older continuity of 
geodesic balls in the interior of segments and Cheeger-Naber \cite{ChN} proved the codimension 4 conjecture.
 
 For complete $n$-manifolds $\{(M, x)\}$ with integral Ricci curvature bound, $\bar k(H, p, 1)$ is sufficient small, 
 an interesting question is whether $\{(M, x)\}$ have similar geometric/topological properties or degenerations as manifolds with Ricci curvature lower bound.
 
In \cite{PW1}, Petersen-Wei showed that with integral Ricci curvature bound, the Laplacian comparison and relative volume comparison hold (see \cite{CW} for an improved relative volume comparison).  And then by using similar methods as Cheeger-Colding \cite{CC1, CC2, CC3, CC4},  Petersen-Wei \cite{PW2} and Tian-Zhang \cite{TZ} showed that when the volume of a unit ball has definite positive lower bound, $\volume(B_1(x))\geq v>0$, which is called non-collapsing case, the almost rigidity structure and some degeneration results hold for integral Ricci curvature. In the collapsing case i.e., $\volume(B_1(x))$ can be arbitrary small, Dai-Wei-Zhang \cite{DWZ} derived a local Sobolev constant estimate on manifolds with integral Ricci curvature bound and thus the gradient estimate and maximal principle hold too. A problem that one can not generalize Cheeger-Colding's results to manifolds with integral Ricci curvature bound in the collapsing case similarly as in the non-collapsing case is that there is no general Cheeger-Colding segment inequality in the collapsing case.
 
 In this paper, we will prove the following segment inequality in manifolds with integral Ricci curvature bound which has improved \cite[Proposition 2.30]{TZ} (see Corollary~\ref{seg-one} and Corollary~\ref{seg-two}) and is effective in the collapsing case. In the following, we will always assume $H\leq 0$ for simplicity where when $H>0$, we should assume upper bound of the radius $r<\frac{\pi}{2\sqrt H}$.
 
Consider a complete $n$-manifold $M$ and a function $u: M\to \Bbb R$. For $y, z\in M$, let
 $$\mathcal{F}_u(y, z)=\inf\left\{\int_{\gamma}u, \, \gamma \text{ is a minimal normal geodesic from } y \text{ to } z \right\}.$$

\begin{Thm}[Segment inequality] \label{seg-ine}
Given $n>0, p>\frac{n}2, H$, let $M$ be a complete $n$-manifold. For a geodesic ball $B_r(x)\subset M$, $r<R$, let $A_1, A_2$ be two measurable subsets of $B_r(x)$. Then for any function $u$ in $M$ satisfying that $\left|\mathcal F_u(y, z)\right|\leq C_0$ for each point $(y, z)\in B_{2r}(x)\times B_{2r}(x)$,  the following holds
$$\int_{A_1\times A_2}\left|\mathcal F_u(y, z)\right|dydz \leq 2c(n,H, R)r\left(\volume(A_1)+\volume(A_2)\right)\left(\int_{B_{2r}(x)}|u|+C_0\volume(B_{2r}(x))c(n,p)\bar k^{\frac12}(H, p, 2r)\right),$$ 
where
$$c(n, H, R)=\sup_{0<\frac{t}{2}\leq s\leq t\leq 2R}\frac{\underline{\mathcal{A}}_{H}(t)}{\underline{\mathcal{A}}_{H}(s)},\quad c(n, p)=\left(\frac{(n-1)(2p-1)}{2p-n}\right)^{\frac{1}{2}}$$
and $\underline{\mathcal{A}}_{H}(t)d\theta dt$ is the volume element of the simply connected $n$-space form $\underline{M}_H^n$ of constant curvature $H$.
\end{Thm} 

\begin{Rem} \label{rem-1}
{\rm (1.2.1)} In Cheeger-Colding segment inequality \cite{CC1},  it assumes that $u$ is nonnegative.

{\rm (1.2.2)} For manifolds with integral Ricci curvature bound, we always assume $\bar k(H, p, 1)\leq \delta<\delta(n, p)$. By Lemma~\ref{pack} below, we know that: 

{\rm (1.2.2a)} For each $R\geq1$, $\bar k(H, p, R)\leq B^2(n, H) \bar k(H, p, 1)$, thus when we apply Theorem~\ref{seg-ine} in balls with radius $R\geq 1$, we have
$$\int_{A_1\times A_2}\left|\mathcal F_u(y, z)\right|dydz \leq 2c(n,H, R)r\left(\volume(A_1)+\volume(A_2)\right)\left(\int_{B_{2r}(x)}|u|+C_0\volume(B_{2r}(x))B(n, H)c(n, p)\delta^{\frac12}\right).$$ 

{\rm (1.2.2b)} For each $r<1$, we have $r^2\bar k(H, p, r)\leq \Psi(r | n, p, H)\bar k(H, p, 1)$, thus
$$\int_{A_1\times A_2}\left|\mathcal F_u(y, z)\right|dydz \leq 2c(n,H)\left(\volume(A_1)+\volume(A_2)\right)\left(r\int_{B_{2r}(x)}|u|+C_0\volume(B_{2r}(x)) \Psi(r | n, p, H)\delta^{\frac12})\right),$$ 
where $\Psi$ is a function such that $\Psi\to 0$ if $r \to 0$ and $n, p, H$ fixed. 

Thus we can apply Theorem~\ref{seg-ine} to any balls under the condition $\bar k(H, p, 1)\leq \delta<\delta(n, p)$.
\end{Rem}
By using the segment inequality and Dai-Wei-Zhang's work \cite{DWZ}, as \cite{CC1} we can prove the following almost rigidity structure in manifolds with integral Ricci curvature bound.

\begin{Thm}[Almost rigidity struture] \label{alm-rig}
 Given $n>0, H, p>\frac{n}{2}$, there exist $\delta_0=\delta(n, H, p)>0$ and $\epsilon_0=\epsilon(n, H, p)>0$, such that for any $\delta <\delta_0, \epsilon<\epsilon_0$, if  a complete $n$-manifold $M$ satisfies $\bar k(H, p, 1)<\delta$, then 
 
 {\rm (1.3.1) Almost splitting:} when $H=0$, for any $q_{\pm}\in M$ with $L=d(q_{+}, q_{-})<\delta^{-\min\{\frac{p}{2n}, \frac12\}}$, for $x\in M$ with $d(x, q_{\pm})\geq Lr$, $1\geq r>0$ and $e(x)=d(x, q_+)+d(x, q_-)-L\leq \epsilon$, there exists a length space $X$ such that 
 $$d_{GH}(B_r(x), B_r((0, x^*)))\leq \Psi(\delta, \epsilon, L^{-1} | n, p, r),$$
where $B_r((0, x^*))\subset \Bbb R\times X$.

{\rm (1.3.2) Almost metric cone:} if for $b>0$
\begin{equation}(1+\epsilon)\frac{\volume(\partial B_b(x))}{\svolsp{H}{b}}\geq \frac{\volume(B_b(x))}{\svolball{H}{b}},\label{vol-con}\end{equation}
then for each $b>\alpha>0$, there exists a compact length space $X$ with $\op{diam}(X)\leq (1+\Psi(\epsilon, \delta | n, p, H, b, \alpha))\pi$,
$$d_{GH}(B_{b-\alpha}(x), B_{b-\alpha}((0, x^*)))\leq \Psi(\epsilon, \delta | n, p, H, b, \alpha),$$
where $B_{b-\alpha}((0, x^*))\subset \Bbb R\times_{\op{sn}_H(t)} X$, $\underline{B}_r^H\subset \underline{M}_H^n$,
$$\op{sn}_H(t)=\left\{\begin{array}{cc} \frac{\sin\sqrt H t}{\sqrt H}, & H>0;\\
t, & H=0;\\
\frac{\sinh\sqrt{-H}t}{\sqrt{-H}}, & H<0.\end{array}\right.$$

 \end{Thm}
 \begin{Rem}
 In \cite{CC1}, under the almost volume cone condition 
 \begin{equation}(1+\epsilon)\frac{\volume(A_{a, b}(x))}{\volume(\partial B_a(x))}\geq \frac{\svolann{H}{a, b}}{\svolsp{H}{a}}, \label{ann-con}\end{equation}
 where $a<b$, they derived the annulus metric cone structure.  
 \end{Rem}
 
 A direct corollary of (1.3.1) is that the splitting theorem holds in the limit spaces of integral Ricci curvature bound sequences.  
 \begin{Cor}[Splitting theorem] \label{spl-thm}
 Given $n, p>\frac{n}{2}$, assume $(X, d)$ is a Gromov-Hausdorff limit of a sequence of complete $n$-manifolds $(M_i, x_i)$ with $\bar k_{M_i}(0, p, 1)\leq \delta_i\to 0$. Then if $X$ contains a line, $(X, d)$ splits out a $\Bbb R$ factor isometrically, i.e.
 $$(X, d) \cong (\Bbb R\times Y, d_{\Bbb R}\times d_Y),$$
 where $(Y, d_Y)$ is a length space.
 \end{Cor}
  
 And for a limit space $(X, d)$ as in Corollary~\ref{spl-thm}, as in \cite{CoN}, we have the following sharp H\"older continuity of small balls in the interior of a limit geodesic:
 \begin{Thm}[Sharp H\"older continuity] \label{sha-hol}
Given $n, p>\frac{n}{2}$, there is $\alpha=\alpha(n, p), C(n, p), r_0=r(n, p)>0$ such that if $(X, x, d)$ is a limit space of a sequence of complete $n$-manifolds $(M_i, x_i)$ with $\bar k_{M_i}(-1, p, 1)\leq \delta_i\to 0$ and $\gamma: [0, l]\to X$ is a unit speed limit geodesic of $\gamma_i: [0, l]\to M_i$, then for any fixed small $\sigma>0$, and any $0<r<r_0\sigma l$, $\sigma l<s<t<l-\sigma l$,
$$d_{GH}(B_r(\gamma(s)), B_r(\gamma(t)))\leq \frac{C(n, p)}{\sigma l}r|s-t|^{\alpha}.$$
\end{Thm}
 By this sharp H\"older continuity, as in \cite{CoN}, we can derive some structure results about the regularity of the limit spaces (see Theorem~\ref{reg-str}) and show that the isometric group of $(X, d)$ is a Lie group (see Theorem~\ref{lie-gro}).

The paper is organized as follows. In section 2, we will supply some preliminaries about manifolds with integral Ricci curvature bound. In section 3, we will give the proof of our main results about the segment inequality and almost rigidity structure. In section 4, we generalize the sharp H\"older continuity of the balls in the interior of a segment for the limit spaces of manifolds with integral Ricci curvature bound and give some regularity structure of the limit spaces.  By using these results, we will see that the isometric groups of the limit spaces are Lie groups. 
\section{Preliminaries}
In this section, we will supply some notions and properties we need in manifolds with integral Ricci curvature bound.

For a complete $n$-manifold $M$, $x\in M$, let $M\setminus \op{Cut}_x$ be equipped with the polar coordinate and let $\mathcal{A}(t, \theta)d\theta dt$ be the volume element. Let $\mathcal A(t, \theta)=0$ when $t$ increases and $\mathcal A(t, \theta)$ is undefined. Let $r=d(x, \cdot)$ be the distance function from $x$ and let 
$\psi=\max\{\Delta r-\underline{\Delta}_H r, 0\}$, where $\underline{\Delta}_H$ is the Laplacian operator in the simply connected space $\underline{M}_H^n$ with constant sectional curvature $H$. Then $\psi=0$ if $\op{Ric}_M\geq (n-1)H$. 
\begin{Thm}  \label{com-int}
Give $n, p> \frac{n}{2}$,$r>0, H$, for a complete $n$-manifold $M$, fix $x \in M$, 
then the following holds:

{\rm (2.1.1)} Laplacian comparison estimates \cite[Lemma 2.2]{PW1}: 
\begin{eqnarray}
\int_{0}^{r}\psi^{2p}\mathcal A(t, \theta)dt & \leq & c(n, p)^{2p} \int_{0}^{r} \rho_H^p\, \mathcal A(t, \theta)dt,  \label{lap-com}
\end{eqnarray}
where $c(n, p)= \left(\frac{(n-1)(2p-1)}{2p-n}\right)^{\frac{1}{2}}$;

{\rm (2.1.2)} Relative volume comparison: \cite{PW1, CW}:
There exists $\delta_0=\delta(n, p, H)$, such that if $\bar k(H, p, 1)\leq \delta\leq \delta_0$, then for each $0<r<R$,
\begin{equation}
\frac{\volume(B_R(x))}{\svolball{H}{R}}\leq e^{c(n, p, H)(\max\{R, 1\}-r)\delta^{\frac12}}\frac{\volume(B_r(x))}{\svolball{H}{r}}.
 \label{vol-comp}
\end{equation}

\begin{equation}
\volume(B_r(x))\leq e^{c(n, p, H)\max\{r, 1\}\delta^{\frac12}}\svolball{H}{r}. \label{vol-com}
\end{equation}
\begin{equation}
\frac{\volume(\partial B_r(x))}{\volume(B_r(x))}\leq \frac{\svolsp{H}{r}}{\svolball{H}{r}}+c(n, p)\bar k^{\frac12}(H, p, r).\label{sph-ball}
\end{equation}
\end{Thm}

By the relative volume comparison and a simple packing argument, we have that
\begin{Lem}[\cite{PW2}]
\label{pack}
Given $n, p>\frac{n}{2}, H$, there is $\delta=\delta(n, p, H)>0$ such that if a complete $n$-manifold $M$ satisfies $\bar k(H, p, 1)\leq \delta$, then for any $R\geq 1\geq r>0$, we have that
\begin{equation}\bar k(H, p, R)\leq B(n, H)^2 \, \bar k(H, p, 1),  \label{c-com}\end{equation}
and 
\begin{equation}\bar k_{r^{-1}M}(0, p, 1)\leq \Psi(\delta, r | n, H, p), \label{small}\end{equation}
where $B(n, H) = \left(2\frac{\svolball{H}{1}}{\svolball{H}{\frac{1}{2}}}\right)^{\frac{1}{2p}}$ and $r^{-1}M$ denote $(M, r^{-2}g)$, $g$ is the Riemannian metric of $(M, g)$.
\end{Lem}
\begin{proof}
Take $\delta=\delta(n, p, H)$ such that in \eqref{vol-comp}, for any $0\leq r_1\leq r_2\leq 1$, 
$$\frac{\volume(B_{r_2}(x))}{\svolball{H}{r_2}}\leq 2 \frac{\volume(B_{r_1}(x))}{\svolball{H}{r_1}}.$$

For $R>1$, for any $x\in M$, take a maximal set $\{x_i\in B_R(x)\}$, such that for $i\neq j$, $d(x_i, x_j)\geq 1$. Then
\begin{eqnarray*}
-\kern-1em\int_{B_R(x)} \rho_H^p dv & \leq & \frac{1}{\volume(B_R(x))}\sum_i\int_{B_1(x_i)}\rho_H^p\\
&\leq & \frac{\sum_i\volume(B_1(x_i))}{\volume(B_R(x))}\bar k^p(H, p, 1)\\
&=& \frac{\sum_i\volume(B_{\frac12}(x_i))\frac{\volume(B_1(x_i))}{\volume(B_{\frac12}(x_i))}}{\volume(B_R(x))}\bar k^p(H, p, 1)\\
&\leq & 2\frac{\sum_i\volume(B_{\frac12}(x_i))}{\volume(B_R(x))}\frac{\svolball{H}{1}}{\svolball{H}{\frac12}}\bar k^p(H, p, 1)\\
& \leq & B^{2p}(n, H)\bar k^p(H, p, 1).
\end{eqnarray*}

For $r<1$, since $p>\frac{n}{2}$, 
$$r^2\bar k(H, p, r)\leq r^2\left(\sup_x\frac{\volume(B_1(x))}{\volume(B_r(x))}\right)^{\frac1p}\bar k(H, p, 1)\leq r^2\left(2\frac{\svolball{H}{1}}{\svolball{H}{r}}\right)^{\frac1p}\bar k(H, p, 1)\leq \Psi(r | n, p, H)\delta.$$
And since
$$\bar k_{r^{-1}M}(r^2H, p, 1)=r^2\bar k(H, p, r), \quad \bar k(0, p, 1)\leq \bar k(H, p, 1)-(n-1)H,$$
we derive that
$$\bar k_{r^{-1}M}(0, p, 1) \leq \Psi(\delta, r | n, H, p).$$
\end{proof}

By the relative volume comparison, the set of manifolds with integral Ricci curvature bound is precompact.
\begin{Thm}[\cite{PW1} Precompactness] \label{compact}
For $n\geq 2, p>\frac{n}{2}, H$, there exists $c(n, p, H)$ such that if a sequence of  compact Riemannian $n$-manifold $M_i$ satisfies that 
$\bar{k_i}(H, p, 1)\leq c(p, n,  H)$, then there is a subsequence of $\{(M_i, x_i)\}$ that converges in the pointed Gromov-Hausdorff topology.
\end{Thm}

In \cite{DWZ},  a local Sobolev constant estimate is obtained for integral Ricci curvature. By this Sobolev constant estimate, the following maximal principle, gradient estimate and the mean value inequality hold.

\begin{Thm}[\cite{DWZ} Maximal principle] \label{max-pri}
Given $n>0, p, q> \frac{n}{2}, H$ and $R>0$, there exist $\delta =\delta(n, H, p, q, R)>0$ and $c(n, H, p, q, R)>0$ such that if a complete $n$-manifold $M$ satisfies $\bar k(H, p, R)\leq \delta$, then for $r\leq R$ and for any function $u: \Omega\to \Bbb R$ with $\Delta u\geq f$, 
$$\sup_{\Omega}u\leq \sup_{\partial \Omega}u+ c(n, H, p, q,R) \left(-\kern-1em\int_{\Omega}\left|\max\{0, -f\}\right|^q\right)^{\frac1q},$$
where $\Omega\subset B_R(x)$ has smooth boundary and $\partial \Omega\cap \partial B_R(x)=\emptyset$.
\end{Thm}

\begin{Thm}[\cite{DWZ} Gradient estimate] \label{gra-est}
Let the assumption be as in Theorem~\ref{max-pri} and let $u: B_R(x)\to \Bbb R$ satisfying that $\Delta u=f$. Then
$$\sup_{B_{\frac{R}{2}}(x)}|\nabla u|^2\leq c(n, H, p, R)\left(-\kern-1em\int_{B_R(x)}u^2 + \left(-\kern-1em\int_{B_R(x)}f^{2p}\right)^{\frac1p}\right).$$
\end{Thm}

\begin{Thm}[\cite{DWZ} Mean value inequality] \label{mea-val}
Given $n>0, p>\frac{n}{2}, H$, there exist $\delta(n, p, H)>0$, such that if a complete $n$-manifold satisfies $\bar k(H, p, 1)\leq \delta(n, p, H)$, then for each nonnegative function $u$ in $M\times [0, r^2]$, $r<1$ with 
$$\frac{\partial}{\partial t} u\geq \Delta u-f,$$
where $f$ is a nonnegative function, then for any $q>\frac{n}2$,
$$-\kern-1em\int_{B_{\frac{r}2}(x)} u \leq C(n, p, q, H)\left(u(x, r^2)+r^2\sup_{0\leq t\leq r^2}\left(-\kern-1em\int_{B_r(x)}f^q\right)^{1/q}\right).$$
\end{Thm}
\begin{Cor}[\cite{DWZ}] \label{mea-cor}
Assume as above, then for each nonnegative function $u$ in $M$, $r<1$ with $\Delta u\leq f,$
$$-\kern-1em\int_{B_{\frac{r}2}(x)} u \leq C(n, p, q, H) \left(u(x)+r^2\left(-\kern-1em\int_{B_r(x)}\max\{0, f\}^q\right)^{1/q}\right).$$
\end{Cor}

By maximal principle Theorem~\ref{max-pri}, we have the Cheeger-Colding cut-off function for integral Ricci curvature.

\begin{Lem}[\cite{DWZ} Cut-off function]\label{cut-off1}
Given $n, R>0$, $p>\frac{n}{2}, H$, there exist $\delta=\delta(n, p, H, R), c(n, H, p, R)>0$ such that if $\bar k(H, p, R)\leq \delta$, then for any $x\in M$, there exists $\phi: M\to [0, 1]$, $\left.\phi\right|_{B_{\frac{R}{2}}(x)}=1$, $\phi\in C^{\infty}_0(B_R(x))$ and 
$$R^2|\Delta \phi|+ R|\nabla \phi|\leq c(n, H, p, R).$$
\end{Lem}

\begin{Cor}[\cite{DWZ}]\label{cut-off2}
Given $n, R>0$, $p>\frac{n}{2}, H$, there exist $\delta=\delta(n, p, H, R), c(n, H, p, R)>0$ such that if $\bar k(H, p, R)\leq \delta$, then for any $x\in M$, $0<10r_1<r_2<R$, there exists $\phi: M\to [0, 1]$, $\left.\phi\right|_{A_{3r_1, \frac{r_2}{3}}(x)}=1$, $\phi\in C^{\infty}_0(A_{2r_1, \frac{r_2}{2}}(x))$ and 
$$r_1^2|\Delta \phi|+ r_1|\nabla \phi|\leq c(n, H, p, R), \text{ in }A_{2r_1, 3r_1}(x),$$
$$r_2^2|\Delta \phi|+ r_2|\nabla \phi|\leq c(n, H, p, R), \text{ in }A_{\frac{r_2}{3}, \frac{r_2}{2}}(x).$$
\end{Cor}

\section{Segment inequality and Almost rigidity structure for integral Ricci curvature}

In this section, we will give the proof of Theorem~\ref{seg-ine}. And then Theorem~\ref{alm-rig} will follow by using Theorem~\ref{seg-ine}, Theorem~\ref{max-pri}, Theorem~\ref{gra-est} and Lemma~\ref{cut-off1} and a similar argument as in \cite{CC1} (see \cite{PW2, TZ} for the non-collapsing case).

\subsection{Segment inequality}
\begin{proof}[Proof of Theorem 1.1]

For each $y\in A_1$, let 
$$\Omega_y=\{(t,\theta)\in T_yM, \, \exp_y(t\theta)\in A_2\setminus\op{Cut}_y, d(y, \exp_y(t\theta))=t\}.$$
\begin{eqnarray*}
& & \int_{A_2}\left|\mathcal{F}_u(y, z)\right|d z \\
&=& \int_{A_2\setminus\op{Cut}_y} \left|\mathcal{F}_u(y, z)\right|d z\\
&=& \int_{(t, \theta)\in \Omega_y}\left|\int_0^{t}u(\exp_y(s\theta))ds\right|\mathcal{A}(t, \theta)dt d\theta\\
&\leq & \int_{(t, \theta)\in \Omega_y}\left|\int_0^{\frac{t}2}u(\exp_y(s\theta))ds\right|\mathcal{A}(t, \theta)dt d\theta+\int_{(t, \theta)\in \Omega_y}\left|\int_{\frac{t}2}^{t}u(\exp_y(s\theta))ds\right|\mathcal{A}(t, \theta)dt d\theta.\\
\end{eqnarray*}
Let 
$$\mathcal{F}_1(y, z)=\left|\int_0^{\frac{t}{2}}u(\exp_y(s\theta))ds\right|, \quad \mathcal{F}_2(y, z)=\left|\int_{\frac{t}2}^{t}u(\exp_y(s\theta))ds\right|.$$
Since 
$$\frac{d}{dr}\frac{\mathcal A(r, \theta)}{\underline{\mathcal{A}}_H(r)}\leq \psi(r, \theta)\frac{\mathcal{A}(r, \theta)}{\underline{\mathcal{A}}_H(r)},$$
for $t\leq r$,
\begin{equation}\frac{\mathcal A(r, \theta)}{\underline{\mathcal{A}}_H(r)}-\frac{\mathcal A(t, \theta)}{\underline{\mathcal{A}}_H(t)}\leq \int_t^r\psi(\tau, \theta)\frac{\mathcal{A}(\tau, \theta)}{\underline{\mathcal{A}}_H(\tau)}d\tau. \label{vol-element}\end{equation}

Then for any fixed $(t, \theta)\in \Omega_y$ and for each $s_0\in [\frac{t}2, t]$,
\begin{eqnarray*}
& &\left|\int_{\frac{t}2}^{t}u(\exp_y(s\theta))ds\right|\mathcal A(t, \theta)\\
&\leq & \left|\int_{\frac{t}{2}}^{t}u(\exp_y(s\theta))ds\right|\left(\mathcal{A}(s_0, \theta)\frac{\underline{\mathcal A}_H(t)}{\underline{\mathcal A}_H(s_0)}+\underline{\mathcal A}_H(t)\int_{s_0}^t\psi\frac{\mathcal A(\tau, \theta)}{\underline{\mathcal A}_H(\tau)}d\tau\right)\\
&\leq & \left|\int_{\frac{t}{2}}^{t}u(\exp_y(s\theta))ds \right|\frac{\underline{\mathcal A}_H(t)}{\underline{\mathcal A}_H(s_0)}\left(\mathcal{A}(s_0, \theta)+\int_{s_0}^t\psi\mathcal A(\tau, \theta)d\tau\right)\\
&\leq & \sup_{\frac{t}{2}\leq s\leq t\leq 2R}\frac{\underline{\mathcal{A}}_{H}(t)}{\underline{\mathcal{A}}_{H}(s)}\left(\left|\int_{\frac{t}2}^tu(\exp_y(s\theta))ds\right|\mathcal{A}(s_0, \theta) +\left|\int_{\frac{t}{2}}^{t}u(\exp_y(s\theta))ds\right|\int_{s_0}^t\psi\mathcal A(\tau, \theta)d\tau \right)\\
&\leq & c(n, H, R)\left|\int_{\frac{t}2}^tu(\exp_y(s\theta))ds\right|\mathcal{A}(s_0, \theta) +c(n, H, R)C_0\int_{\frac{t}2}^t\psi\mathcal A(\tau, \theta)d\tau\\
&\leq & c(n, H, R)\int_{\frac{t}2}^t\left|u(\exp_y(s\theta))\right|ds\mathcal{A}(s_0, \theta) +c(n, H, R)C_0\int_{\frac{t}2}^t\psi\mathcal A(\tau, \theta)d\tau,
\end{eqnarray*}
where $$c(n, H, R)=\sup_{0<\frac{t}{2}\leq s\leq t\leq 2R}\frac{\underline{\mathcal{A}}_{H}(t)}{\underline{\mathcal{A}}_{H}(s)}.$$
Thus
\begin{eqnarray*}
& & \left|\int_{\frac{t}2}^{t}u(\exp_y(s\theta))ds\right|\mathcal A(t, \theta)\\
&\leq&  c(n, H, R)\int_{\frac{t}2}^t\left|u(\exp_y(s\theta))\right|\inf_{\frac{t}{2}\leq s_0\leq t}\mathcal{A}(s_0, \theta)ds +c(n, H, R)C_0\int_{\frac{t}2}^t\psi\mathcal A(\tau, \theta)d\tau\\
&\leq & c(n, H, R)\int_{\frac{t}2}^t\left|u(\exp_y(s\theta))\right|\mathcal{A}(s, \theta)ds +c(n, H, R)C_0\int_{\frac{t}2}^t\psi\mathcal A(\tau, \theta)d\tau.
\end{eqnarray*}

And then
\begin{eqnarray}
& &\int_{A_2}\mathcal{F}_2(y, z) dz\\
&=& \int_{(t, \theta)\in \Omega_y}\left|\int_{\frac{t}{2}}^{t}u(\exp_y(s\theta))ds\right|\mathcal{A}(t, \theta)dt d\theta \nonumber\\
&\leq & c(n, H, R)\int_{(t, \theta)\in \Omega_y}\left(\int_{\frac{t}2}^t\left|u(\exp_y(s\theta))\right|\mathcal{A}(s, \theta)ds +C_0\int_{\frac{t}2}^t\psi\mathcal A(\tau, \theta)d\tau\right) dt d\theta \nonumber\\
&\leq & c(n, H, R)\int_{S^{n-1}}\int_0^{2r} \left(\int_{\frac{t}2}^t\left|u(\exp_y(s\theta))\right|\mathcal{A}(s, \theta)ds +C_0\int_{\frac{t}2}^t\psi\mathcal A(\tau, \theta)d\tau\right)dt d\theta, \nonumber\\
&\leq & c(n, H, R) \cdot 2r \left(\int_{B_{2r}(x)}|u| + C_0\int_{B_{2r}(x)}\psi \right), \nonumber
\end{eqnarray}
$$\int_{A_1\times A_2}\mathcal{F}_2(y, z) dv_y dv_z\leq 2c(n,H, R)r\volume(A_1)\left(\int_{B_{2r}(x)}|u| + C_0\int_{B_{2r}(x)}\psi\right).$$
Similarly, we have that 
$$\int_{A_1\times A_2}\mathcal{F}_1(y, z) dv_y dv_z\leq 2c(n, H, R)r\volume(A_2)\left(\int_{B_{2r}(x)}|u|+ C_0\int_{B_{2r}(x)}\psi\right).$$

Finally, using H\"older inequality and Laplacian comparison \eqref{lap-com}, we derive the result.
\end{proof}

If $u$ is non-negative and has $C^0$ bound, then Theorem 1.1 is just \cite[Proposition 2.29]{TZ}.
\begin{Cor} \label{seg-one}
If $\|u\|_{C^0(B_R(x))}\leq C_0$, $u\geq 0$, then
$$\int_{A_1\times A_2}\mathcal{F}_u(y, z) \leq 2c(n,H, R)r(\volume(A_1)+\volume(A_2))\left(\int_{B_{2r}(x)}u+rC_0\int_{B_{2r}(x)}\psi\right).$$ 
\end{Cor}

Comparing with \cite[Proposition 2.30]{TZ}, we have that
\begin{Cor}\label{seg-two}
 Assume $u\in C^{\infty}(B_{R}(x))$ satisfies that  $|\nabla u|\leq \Lambda$. Then for any unit speed geodesic $\gamma$ from $y\in A_1$ to $z\in A_2$, the following holds:
$$\int_{A_1\times A_2}|\left<\nabla u, \gamma'\right>(y) - \left<\nabla u, \gamma'\right>(z)|\leq 2c(n, H, R)r\left(\volume(A_1)+\volume(A_2)\right)\left(\int_{B_{2r}(x)}|\op{Hess}_u|+ 2\Lambda \int_{B_{2r}(x)}\psi\right).$$
\end{Cor}
\begin{proof}
Let $f(\gamma(s))=\op{Hess}_u(\gamma', \gamma')(\gamma(s))$. Then by $|\mathcal F_f(y, z)|=|\left<\nabla u, \gamma'\right>(\gamma(0)) - \left<\nabla u, \gamma'\right>(\gamma(d(y,z)))|\leq 2\Lambda$,  the result is derived by the proof of Theorem~\ref{seg-ine}.
\end{proof}

Consider a sequence of complete $n$-manifolds $(M_i, x_i)$ with $\bar k_i(H, p, 1)\leq \delta_i\to 0$.  By the precompactness (Theorem~\ref{compact}) and a similar argument as in \cite{CC2}, passing to a subsequence, there is a metric measure space $(X, x, d, \nu)$, such that
$\left(M_i, x_i, g_i, \frac{\volume(\cdot)}{\volume(B_1(x_i))}\right)$ is measured Gromov-Hausdorff convergent to $(X, x, d, \nu)$, where $\nu$ is a Radon measure and satisfies that for $y\in X$, $0<r_1<r_2$,
$$\frac{\nu(B_{r_1}(y))}{\nu(B_{r_2}(y))}\geq \frac{\svolball{H}{r_1}}{\svolball{H}{r_2}}.$$
And $(X, d, \nu)$ is called a limit space. In $(X, d, \nu)$, a limit geodesic $\gamma: [0, l]\to X$ is a geodesic which is a limit of a sequence of geodesics $\gamma_i:[0, l_i]\to M_i$, i.e., $l_i\to l$ and $\gamma_i\to \gamma$. 
Then as in \cite{CC4}, we have the segment inequality in the limit spaces.

\begin{Cor} \label{seg-lim}
Let $(X, x, d, \nu)$ be as the above. Then for any geodesic ball $B_r(x)$, $r<R$, measurable subsets of $B_r(x)$, $A_1, A_2$ and a function $u\in L^q(X, \nu)$, $(q>1)$,  satisfying that $\left|\mathcal F_u(y, z)\right|$ is uniformly bounded for any points $(y, z)\in B_{2r}(x)\times B_{2r}(x)$,  the following holds
$$\int_{A_1\times A_2}\left|\mathcal F_u(y, z)\right|dydz \leq 2c(n,H, R)r\left(\nu(A_1)+\nu(A_2)\right)\int_{B_{2r}(x)}|u|,$$ 
where $\mathcal F_u(y, z)=\inf\left\{\int_{\gamma} u(\gamma(t)), \, \gamma \text{ is a limit geodesic connecting } y, z\right\}$.
\end{Cor}
\begin{proof} Assume $\left|\mathcal F_u(y, z)\right|\leq C_0$.
Write $u=u_+-u_{-}$, where $u_+=\max\{0, u\}$ and $u_{-}=\max\{0, -u\}$. Then as the discussion in \cite[Theorem 2.6]{CC4}, there are nondecreasing continuous functions: $h_j\to u_+$, $f_j\to u_{-}$ and 
$$\mathcal F_{u_+}=\lim_{j\to \infty}\mathcal F_{h_j}, \quad \mathcal F_{u_{-}}=\lim_{j\to \infty} \mathcal F_{f_j}.$$
Then continuous functions 
$$h_j-f_j\to u, \quad \mathcal F_u=\lim_{j\to \infty} \mathcal F_{h_j-f_j},$$
and for $j$ large, $\left|\mathcal F_{h_j-f_j}\right|\leq 2C_0$.
Now without loss of generality, we may assume $u$ is continuous. Then by using \eqref{vol-com}, \eqref{c-com} and a straightforward limit argument as in \cite{CC4}, we derive the result.
\end{proof}

\subsection{Almost rigidity results}
To prove the almost rigidity results Theorem~\ref{alm-rig}, first recall the following excess estimate.
\begin{Thm}[\cite{PW2} Excess estimate] \label{exc-est}
Given $n, 1\geq r>0, p>\frac{n}{2}$, there exist $\epsilon(n, p, r), \delta(n, p, r)>0$, such that for $\epsilon <\epsilon(n, p, r), \delta \leq \delta(n, p, r)$, if a complete $n$-manifold $M$ satisfies that for $q_{+}, q_{-}, x\in M$,
$$\bar k(0, p, 1)\leq \delta, \quad d(q_+, q_-)=L<\delta^{-\min\{\frac{p}{2n}, \frac12\}}, \quad d(q_+, x), d(q_-, x)\geq Lr, \quad e(x)\leq \epsilon,$$
where $e(x)=d(q_+, x)+d(q_-, x)-L$,
then for each $y\in B_{\frac{r}{2}}(x)$,
$$e(y)\leq \Psi(\epsilon, \delta, L^{-1} | n, p, r).$$
\end{Thm}
Note that in \cite{PW2}, the excess estimate is proved in the non-collapsing case. Here we write a proof by using Corollary~\ref{mea-cor} (see also \cite[Theorem 5.6]{DWZ}).
\begin{proof}
Let $\psi_{\pm}=\max\{\Delta d(\cdot, q_{\pm})- \underline{\Delta}d(\cdot, q_{\pm}), 0 \}$.
Then for $y\in B_r(x)$,
$$\Delta e(y)\leq \frac{2(n-1)}{Lr-r}+ \psi_+ + \psi_-.$$
By Corollary~\ref{mea-cor}, 
\begin{eqnarray*}
-\kern-1em\int_{B_{\frac{r}2}(x)}e(y)&\leq& c(n, p)\left(e(x)+r^2\left(-\kern-1em\int_{B_r(x)}\left(\frac{2(n-1)}{Lr-r}+ \psi_+ + \psi_- \right)^{2p}\right)^{1/2p}\right)\\
&\leq & c'(n, p)\left(e(x)+r^2\left(\frac{2(n-1)}{Lr-r}+ \left(-\kern-1em\int_{B_r(x)}\psi^{2p}_+\right)^{1/2p} + \left(-\kern-1em\int_{B_r(x)}\psi^{2p}_- \right)^{1/2p}\right)\right).
\end{eqnarray*}

We know that 
\begin{eqnarray*} \left(-\kern-1em\int_{B_r(x)}\psi_{\pm}^{2p}\right)^{\frac1{2p}} & \leq &\left(\frac1{\volume(B_r(x))}\int_{B_{L+2r}(q_{\pm})}\psi^{2p}_{\pm}\right)^{\frac{1}{2p}}\\
&\leq & c(n, p) \left(\frac1{\volume(B_r(x))}\int_{B_{L+2r}(q_{\pm})}\rho_H^{p}\right)^{\frac{1}{2p}}\\
&=& c(n, p)\left(\frac{\volume(B_{2L+2r}(x))}{\volume(B_r(x))}\right)^{\frac1{2p}} \bar k^{\frac12}(0, p, L+2r)\\
&\leq & c(n, p)e^{c'(n, p)L\delta^{\frac12}}\left(\frac{2L+2r}{r}\right)^{\frac{n}{2p}}\delta^{\frac12}\\
&\leq & c''(n, p)\left(\frac{2L+2r}{r}\right)^{\frac{n}{2p}}\delta^{\frac12}\\
& < & C(n, p)\delta^{\frac14},
\end{eqnarray*}
where we use \eqref{vol-comp} and the condition $L<\delta^{-\min\{\frac12, \frac{p}{2n}\}}$.
Thus
$$-\kern-1em\int_{B_{\frac{r}2}(x)}e(y)\leq \Psi(\epsilon, \delta, L^{-1} | n, p, r).$$
Since $e$ is 2-Lipschitz, we have that for $y\in B_{\frac{r}2}(x)$,
$$e(y)\leq \Psi(\epsilon, \delta, L^{-1} | n, p, r).$$
\end{proof}

Let $h_{\pm}$ be the harmonic functions on $B_r(x)$ with $\left.h_{\pm}\right|_{\partial B_r(x)}=\left.b_{\pm}\right|_{\partial B_r(x)}$ where $b_{\pm}=d(q_{\pm}, \cdot)-d(q_{\pm}, x)$. Then as in \cite{PW2}, we have that
\begin{Lem} Let the assumption be as in Theorem~\ref{exc-est}, then
 $$|h_{\pm}-b_{\pm}|\leq \Psi(\epsilon, \delta, L^{-1} | n, p, r);$$
 $$-\kern-1em\int_{B_{\frac{r}2}(x)} \left|\nabla h_{\pm}-\nabla b_{\pm}\right|\leq \Psi(\epsilon, \delta, L^{-1} | n, p, r);$$
$$-\kern-1em\int_{B_{\frac{r}2}(x)}|\op{Hess}_{h_{\pm}}|^2\leq \Psi(\epsilon, \delta, L^{-1} | n, p, r).$$
 \end{Lem}
Now just following the proof in \cite{CC1} (see also \cite{Ch}) by using the segment inequality Theorem~\ref{seg-ine}, we can prove the almost splitting theorem (1.3.1).  Note that to apply our segment inequality Theorem~\ref{seg-ine} to test the Pythagorean theorem, we use $\left|\int_{\gamma}\op{Hess}_{h_{\pm}}(\gamma', \gamma')\right|$ instead of $\int_{\gamma}|\op{Hess}_{h_{\pm}}|$.

To prove the almost metric cone structure, we just need to show 
\begin{Lem} \label{alm-lem}
Let the assumption be as in (1.3.2), then there is $\tilde f$ such that for $r(y)=d(x, y)$, 
$$\left|\tilde f-f \right|\leq \Psi,  \text{ on }B_{b-\Psi}(x),\quad C^0-\text{estimate},$$
$$ -\kern-1em\int_{B_{b-\Psi}(x)} \left|\nabla \tilde f -\nabla f\right|^2\leq \Psi, \quad C^1-\text{estimate},$$
$$-\kern-1em\int_{B_{b-\alpha}(x)} \left|\op{Hess}_{\tilde f}-f''\right|^2\leq \Psi, \quad C^2-\text{estimate},$$
where 
$$f=-\int_t^b\op{sn}_H(s)ds,$$
and
$\Psi=\Psi(\epsilon, \delta | n, p, H, b)$ or $\Psi(\epsilon, \delta | n, p, H, b, \alpha)$.
\end{Lem}
The proof of the above lemma is similar as in \cite{CC1} (see also \cite{Ch}). In fact, take $\tilde f$ as
$$\left\{\begin{array}{cc}\Delta \tilde f = n \op{sn}'_H(r), & \text{ in } B_{b}(x);\\
\tilde f= f, &  \text{ on } \partial B_{b}(x).
\end{array}\right.$$

The key point is the Laplacian estimate of $\tilde f$. For $H=0$, 
$$ -\kern-1em\int_{B_b(x)} \Delta \tilde f  =   -\kern-1em\int_{B_b(x)} n = n.$$
And by Laplacian comparison,
\begin{eqnarray*}
\int_{B_b(x)} \Delta \tilde f &= & \lim_{\sigma \to 0}\int_{B_b(x)\setminus U_{\sigma}} \underline{\Delta} \frac{r^2}{2} \\
&\geq & \lim_{\sigma \to 0}\int_{B_b(x)\setminus U_{\sigma}} \Delta \frac{r^2}2 - r \psi\\
&\geq  & \lim_{\sigma \to 0}\int_{\partial U_{\sigma}\cap B_b(x)} \langle\nabla \frac{r^2}2, N\rangle + \int_{\partial B_b(x)} \langle\nabla \frac{r^2}2, N\rangle - b\int_{B_b(x)}\psi\\
&\geq &  b \volume(\partial B_b(x)) - b\int_{B_b(x)}\psi,
\end{eqnarray*}
where $C$ is the cut locus of $x$ and $U_{\sigma}\subset B_{\sigma}(C\setminus \{x\})$ which has piecewise smooth boundary, $N$ is taken such that $\left<\nabla r ,N\right>>0$ and without loss of generality we may assume $\volume(\partial B_b(x)\setminus C)=\volume(\partial B_b(x))$.

Then
\begin{eqnarray*}
-b \bar k^{\frac12}(H, p, b) &\leq & -\kern-1em\int_{B_b(x)}\Delta \tilde f - \Delta \left(\frac{r^2}2\right) \\
&\leq & \left(n- b\frac{\volume(\partial B_b(x))}{\volume(B_b(x))}\right)\\
&=& b\left(\frac{\svolsp{0}{b}}{\svolball{0}{b}}- \frac{\volume(\partial B_b(x))}{\volume(B_b(x))}\right)\\
&\leq & \frac{b\epsilon}{1+\epsilon}\frac{\svolsp{0}{b}}{\svolball{0}{b}}.
\end{eqnarray*}
And thus
$$\left|-\kern-1em\int_{B_b(x)}\Delta \tilde f - \Delta \left(\frac{r^2}2\right)\right|\leq \Psi(\epsilon, \delta | n, p, b).$$

For $H\neq 0$, first note that \eqref{vol-con} implies that for each $0<a<b$, 
\begin{equation}\frac{\volume(\partial B_a(x))}{\svolsp{H}{a}}\leq \left(1+\epsilon\frac{\svolball{H}{b}}{\svolball{H}{a}}\right)\frac{\volume(\partial B_b(x))}{\svolsp{H}{b}}+\frac{\volume(B_a(x))}{\svolsp{H}{a}}c(n, p)\bar k^{\frac12}(H, p, a)+c(n, p)b\frac{\volume(B_b(x))}{\svolball{H}{a}}\bar k^{\frac12}(H, p, b).\end{equation}
In fact,
\begin{eqnarray*}
& & \frac{\volume(\partial B_a(x))}{\svolsp{H}{a}}\frac{\svolsp{H}{b}}{\volume(\partial B_b(x))}\leq \left(\frac{\volume(B_a(x))}{\svolball{H}{a}}+\frac{\volume(B_a(x))}{\svolsp{H}{a}}c(n,p)\bar k^{\frac12}(H, p, a)\right)\frac{\svolsp{H}{b}}{\volume(\partial B_b(x))}\\
& =& \left(\frac{\volume(B_b(x))-\volume(A_{a, b}(x))}{\svolball{H}{a}}+\frac{\volume(B_a(x))}{\svolsp{H}{a}}c(n,p)\bar k^{\frac12}(H, p, a)\right)\frac{\svolsp{H}{b}}{\volume(\partial B_b(x))}\\
&=& \frac{\volume(B_b(x))}{\svolball{H}{a}}\frac{\svolsp{H}{b}}{\volume(\partial B_b(x))}-\frac{\volume(A_{a,b}(x))}{\svolball{H}{a}}\frac{\svolsp{H}{b}}{\volume(\partial B_b(x))}+\frac{\volume(B_a(x))}{\svolsp{H}{a}}c(n,p)\bar k^{\frac12}(H, p, a)\frac{\svolsp{H}{b}}{\volume(\partial B_b(x))}\\
&\leq & (1+\epsilon)\frac{\svolball{H}{b}}{\svolball{H}{a}}-\frac{\svolann{H}{a, b}}{\svolball{H}{a}}+\left(\frac{\volume(B_a(x))}{\svolsp{H}{a}}c(n,p)\bar k^{\frac12}(H, p, a)+\frac{\volume(B_b(x))}{\svolball{H}{a}}c(n,p)b\bar k^{\frac12}(H, p, b)\right)\frac{\svolsp{H}{b}}{\volume(\partial B_b(x))}\\
&=& 1+\epsilon\frac{\svolball{H}{b}}{\svolball{H}{a}}+\left(\frac{\volume(B_a(x))}{\svolsp{H}{a}}c(n,p)\bar k^{\frac12}(H, p, a)+\frac{\volume(B_b(x))}{\svolball{H}{a}}c(n,p)b\bar k^{\frac12}(H, p, b)\right)\frac{\svolsp{H}{b}}{\volume(\partial B_b(x))}.
\end{eqnarray*}
where we use \eqref{sph-ball} and  the fact that by \eqref{vol-element}, for $a\leq s\leq t\leq b$,
$$\frac{\volume(\partial B_t(x))}{\svolsp{H}{t}}-\frac{\volume(\partial B_s(x))}{\svolsp{H}{s}}\leq \frac{\volume(B_t(x))}{\svolsp{H}{s}}c(n,p)\bar k^{\frac12}(H, p, t)$$
and thus
$$\frac{\volume(A_{a, b}(x))}{\svolann{H}{a, b}}\geq \frac{\volume(\partial B_b(x))}{\svolsp{H}{b}}-\frac{\volume(B_b(x))}{\svolann{H}{a, b}}c(n,p)b\bar k^{\frac12}(H, p, b).$$

From the poof of (3.3), we can see that
\begin{equation}
\frac{\volume(B_a(x))}{\svolball{H}{a}}\frac{\svolsp{H}{b}}{\volume(\partial B_b(x))}\leq 1+\epsilon\frac{\svolball{H}{b}}{\svolball{H}{a}}+\frac{\volume(B_b(x))}{\svolball{H}{a}}\frac{\svolsp{H}{b}}{\volume(\partial B_b(x))}c(n,p)b\bar k^{\frac12}(H, p, b).\end{equation}

And then
\begin{eqnarray}
\frac{\volume(B_a(x))}{\volume(B_b(x))} &\leq& \left(\svolball{H}{a}+\epsilon\svolball{H}{b}\right)\frac{\volume(\partial B_b(x))}{\svolsp{H}{b}\volume(B_b(x))} + c(n,p)b\bar k^{\frac12}(H, p, b) \nonumber\\
&\leq & \left(\svolball{H}{a}+\epsilon\svolball{H}{b}\right)\left(\frac{1}{\svolball{H}{b}}+\frac{c(n, p)\bar k^{\frac12}(H, p, b)}{\svolsp{H}{b}}\right)+ c(n,p)b\bar k^{\frac12}(H, p, b) \nonumber\\
&\leq & \Psi(a, \delta, \epsilon | n, p, H, b).
\end{eqnarray}

For $H=1$. Take $0<a<b$,
\begin{eqnarray*}
 -\kern-1em\int_{B_b(x)} \Delta \tilde f & = &  \frac{1}{\volume(B_b(x))}\left(\int_{B_{a}(x)}n\cos r +\int_{A_{a, b}(x)} n\cos r \right)\\
 &\leq &\frac{1}{\volume(B_{b}(x))}\left(n\volume(B_{a}(x))+ \int_a^b\int_{S^{n-1}} n\cos r \mathcal A(r, \theta)d\theta dr\right)\\
&\leq & \frac{1}{\volume(B_{b}(x))}\left(n\volume(B_{a}(x))+ \int_a^b\int_{S^{n-1}} n\cos r \left(\mathcal A(a, \theta)\frac{\underline{\mathcal A}_H(r)}{\underline{\mathcal A}_H(a)} +   \underline{\mathcal A}_H(r)\int_a^r \psi(s, \theta)\frac{\mathcal A(s, \theta)}{\underline{\mathcal A}_H(s)}ds\right)d\theta dr\right)\\
&=& \frac{1}{\volume(B_{b}(x))}\left(n\volume(B_{a}(x))+\frac{\sin^n b-\sin^n a}{\sin^{n-1}a}\volume(\partial B_a(x))+\int_{S^{n-1}}\int_a^b \frac{\sin^n b-\sin^n s}{\sin^{n-1}s}\psi\mathcal A(s, \theta)dsd\theta\right)\\
&\leq & \frac{1}{\volume(B_{b}(x))}\left(n\volume(B_{a}(x))+\frac{\sin^n b-\sin^na}{\sin^{n-1}a}\volume(\partial B_a(x))+\left(\frac{\sin^n b}{\sin^{n-1} a}-\sin a\right)\int_{B_b(x)}\psi\right)
\end{eqnarray*}

By Laplacian comparison,
\begin{eqnarray*}
\int_{B_{b}(x)} \Delta \tilde f &= & \lim_{\sigma \to 0}\int_{B_{ b}(x)\setminus U_{\sigma}} -\underline{\Delta} \cos r \\
&\geq & \lim_{\sigma \to 0}\int_{B_{b}(x)\setminus U_{\sigma}} -\Delta \cos r -\sin r \psi\\
&\geq  & \lim_{\sigma \to 0}\int_{\partial U_{\sigma}\cap B_{b}(x)} \langle-\nabla \cos r, N\rangle + \int_{\partial B_b(x)} \langle-\nabla \cos r, N\rangle -\sin b\int_{B_b(x)}\psi\\
&\geq & \sin b \volume(\partial B_b(x)) -\sin b\int_{B_b(x)}\psi.
\end{eqnarray*}

Then
\begin{eqnarray*}
& & -\sin b c(n,p) \bar k^{\frac12}(H, p, b)\\
& \leq & -\kern-1em\int_{B_{b}(x)}\Delta \tilde f - \Delta (-\cos r) \\
&\leq & \sin b\frac{\svolsp{H}{b}}{\volume(B_{b}(x))}\left( \frac{\volume(\partial B_a(x))}{\svolsp{H}{a}}-\frac{\volume(\partial B_b(x))}{\svolsp{H}{b}}\right) + n\frac{\volume(B_a(x))}{\volume(B_b(x))}+\frac{\sin^n b}{\sin^{n-1} a}c(n,p)\bar k^{\frac12}(H, p, b)\\
&\leq & \sin b\frac{\svolsp{H}{b}}{\volume(B_{b}(x))}\left(\epsilon\frac{\volume(\partial B_b(x))}{\svolsp{1}{b}}\frac{\svolball{H}{b}}{\svolball{H}{a}}+\frac{\volume(B_a(x))}{\svolsp{H}{a}}c(n,p)\bar k^{\frac12}(H, p, a)+\frac{\volume(B_b(x))}{\svolball{H}{a}}c(n,p)b\bar k^{\frac12}(H, p, b)\right)\\
& & + n\frac{\volume(B_a(x))}{\volume(B_b(x))}+\frac{\sin^n b}{\sin^{n-1} a}\bar k^{\frac12}(H, p, b)\\
&\leq & \sin b\epsilon \left(\frac{\svolsp{H}{b}}{\svolball{H}{a}}+c(n, p)\frac{\svolball{H}{b}}{\svolball{H}{a}}\bar k^{\frac12}(H, p, b)\right)\\
& &+ \sin bc(n,p) \frac{\svolsp{H}{b}}{\svolball{H}{a}}\left(\frac{\svolball{H}{a}}{\svolsp{H}{a}}\bar k^{\frac12}(H, p, a)+b\bar k^{\frac12}(H, p, b)\right)
 + n\frac{\volume(B_a(x))}{\volume(B_b(x))}+\frac{\sin^n b}{\sin^{n-1} a}c(n,p)\bar k^{\frac12}(H, p, b)\\
&\leq & \Psi(\epsilon, \delta |n, p,  b),
\end{eqnarray*}
where we take $a=\max\{\epsilon^{\frac1{2n}}, \delta^{\frac{p}{2((n-1)2p+n)}}\}$ and use (3.3), (3.5).

 For $H=-1$, it is similar as the case $H=1$.

Now by a standard argument as in \cite{CC1, Ch},  we can show that $\tilde f$ satisfies Lemma~\ref{alm-lem}. First, using maximal principle Theorem~\ref{max-pri}, derive $|\tilde f-f|\leq c(n, p, a, b)$ and using integral by part, we have $C^1$-estimate by the above Laplacian estimate.  And by gradient estimate Theorem~\ref{gra-est}, we have the upper bound, $|\nabla \tilde f-\nabla f|\leq c(n, p, a, b)$. Using segment inequality Theorem~\ref{seg-ine} and $C^1$-estimate, the $C^0$-estimate is derived. Finally apply cut-off function Theorem~\ref{cut-off1} and Bochner's formula, we have the $C^2$-estimate. 

As the discussion in almost splitting (1.3.1), to prove the cosine law of (1.3.2) by using Lemma~\ref{alm-lem} and Theorem~\ref{seg-ine}, we will use $\left|\int_{\gamma}\op{Hess}_{\tilde f}(\gamma', \gamma')\right|$ instead of $\int_{\gamma}\left|\op{Hess}_{\tilde f}\right|$.

Finally, recall that a $\epsilon$-splitting map $b=\left(b_1, \cdots, b_k\right): B_r(x)\to \Bbb R^k$ is a harmonic map such that
\begin{equation*}\left|\nabla b_j\right|\leq c(n), \forall\,  j;\end{equation*}
\begin{equation*}-\kern-1em\int_{B_r(x)} \sum_{j, l}\left|\left<\nabla b_j, \nabla b_l\right>-\delta_{jl}\right|^2+r^2\sum_j\|\op{Hess}_{b_j}\|^2\leq \epsilon^2.\end{equation*}

And we have that
\begin{Thm}\label{spl-map}
Given $n>0, p>\frac{n}{2}$, consider a sequence of complete $n$-manifolds $(M_i, x_i)\to (X, x)$ with $\bar k_{M_i}(0, p, 1)\leq \delta_i\to 0$ then

{\rm (3.7.1)} If there is a sequence of $\epsilon_i$-splitting maps $b^i: B_2(x_i)\to \Bbb R^k$, $\epsilon_i\to 0$, then
$$(B_1(x_i), x_i)\to (B_1(0, y), (0, y))\subset (\Bbb R^k\times Y, (0, y))$$ where $Y$ is a length space. 

{\rm (3.7.2)} If $(M_i, x_i)\to (\Bbb R^k\times Y, (0, y))$, then for $\epsilon_i\to 0$, any fixed $r>0$, there are $\epsilon_i$-splitting maps $b^i: B_r(x)\to \Bbb R^k$.
\end{Thm}

\section{Sharp H\"older continuity and Regularity of the limit spaces}

By mean value inequality Theorem~\ref{mea-val}, the following parabolic approximation estimate holds (see \cite{TZ}, \cite{ZZ1} for the non-collapsing case and \cite{DWZ}, \cite{ZZ2} for the collapsing case). 
Given two points $q_{\pm}\in M$ with $d(q_+, q_-)=L\geq1$, as in Theorem~\ref{exc-est} let
$$e(x)=d(q_+, x)+d(q_-, x)-d(q_+, q_-),$$
$$b^{+}(x)=d(q_+, q_-)-d(q_+, x), \quad b^-(x)=d(q_-, x), \quad \psi^{\pm}=\max\{0, \Delta b^{\pm}-\underline{\Delta}_Hb^{\pm}\}.$$
Given $\sigma>0$, take cut-off function
$$\phi^{\pm}=\left\{\begin{array}{cc} 1, & \text{ on } A_{\frac{\sigma}{4}L, 8L}(q_{\pm});\\
0, & \text{ on } M\setminus A_{\frac{\sigma}{16}L, 16L}(q_{\pm}),\end{array}\right.$$
as in Corollary~\ref{cut-off2}. Let $\phi=\phi^+\phi^-$ and let $M_{r, s}=A_{rL, sL}(q_{+})\cap A_{rL, sL}(q_-)$.

Let $b^{\pm}_t, e_t$ be the solution to the heat equation $\partial_t-\Delta=0$ on $M$ with $b_0^{\pm}=\phi b^{\pm}$, $e_0=\phi e$. Then $e_t=b^-_t-b^+_t$ and 
\begin{Thm}[\cite{DWZ}]\label{par-app}
Given $n, p>\frac{n}{2}, \sigma>0$, there is $\delta(n, p)>0, C=C(n, p, \sigma)$, such that for $0<\delta<\delta(n, p)$, if $\bar k(-1, p, 1)\leq \delta$, then for $q_{\pm}$ as above with $L<\delta^{-\min\{\frac{p}{2n},\frac12\}}$, $x\in M_{\frac{\sigma}{2}, 4}$, $t<\frac{1}{100}(\sigma L)^2$, 

{\rm (4.1.1)} $-\kern-1em\int_{B_{\sqrt t}(x)}e(y)\leq C\left(e(x)+ L^{-1}t+ t^{1-\frac{n}{4p}}\bar k^{\frac12}(-1, p, L)\right)$.

{\rm (4.1.2)} for $y\in B_{\frac12\sqrt t}(x)$, 
$e_t(y)\leq C\left(e(x)+L^{-1}t+t^{1-\frac{n}{4p}}\delta^{\frac12}\right)$ and 

$|\nabla e_t(y)|\leq Ct^{-\frac12}\left( e(x)+L^{-1}t+t^{1-\frac{n}{4p}}\delta^{\frac12}\right)$;

{\rm (4.1.3)} $\left|b_t^{\pm}-b^{\pm}\right|\leq C\left(e(x)+L^{-1}t+t^{1-\frac{n}{4p}}\delta^{\frac12}\right)$;

{\rm (4.1.4)} $\left|\nabla b_t^{\pm}\right|^2(y)\leq 1+Ct^{1-\frac{n}{4p}}\delta^{\frac12}$ for $y\in M_{\frac{\sigma}{2}, 4}$;


{\rm (4.1.5)} $\int_{t/2}^t-\kern-1em\int_{B_{\sqrt t}(x)}\left|\op{Hess}_{b^{\pm}_{\tau}}\right|^2dvd\tau\leq C\left(e(x)t^{-\frac12}+L^{-1}t^{\frac12}+t^{\frac12-\frac{n}{4p}}\delta^{\frac12}\right)$.
\end{Thm}

And for any $\epsilon$-geodesic $\gamma$, a unit speed curve between $q_+, q_-$ satisfying $\left| |\gamma|-d(q_+, q_-)\right|\leq \epsilon^2 d(q_+, q_-)$, as in \cite{CoN}, by  (4.1.3) and (4.1.4),
\begin{Cor}\label{int-par}
 Let the assumption be as in Theorem~\ref{par-app}, then there is $C=C(n, p, \sigma)$ such that for $\sigma L\leq t_0<t_0+\sqrt t\leq (1-\sigma)L$,
$$\int_{t_0}^{t_0+\sqrt t}-\kern-1em\int_{B_{\sqrt t}(\gamma(s))}\left||\nabla b_t^{\pm}|^2-1\right|dvds\leq C(\epsilon^2L+L^{-1}t+t^{1-\frac{n}{4p}}\delta^{\frac12});$$
$$\int_{\frac{t}2}^t\int_{t_0}^{t_0+\sqrt t}-\kern-1em\int_{B_{\sqrt t}(\gamma(s))}\left|\op{Hess}_{b_{\tau}^{\pm}}\right|^2dvdsd\tau\leq C(\epsilon^2L+L^{-1}t+t^{1-\frac{n}{4p}}\delta^{\frac12}).$$
\end{Cor}

In the following, we will use Theorem~\ref{par-app} and Corollary~\ref{int-par} to prove Theorem~\ref{sha-hol}. 

In \cite{CoN},  to prove the sharp H\"older continuity,  a key point is the better estimates of parabolic approximation (compared with the harmonic approximation in the proof of splitting theorem):
$$-\kern-1em\int_{B_r(\gamma(t))}e\leq C r^2, \quad \int_{\gamma}-\kern-1em\int_{B_r(\gamma(t))} |\op{Hess}_{b^{\pm}_{r^2}}|^2\leq C.$$
In manifolds with integral Ricci curvature, the parabolic approximation estimates in Theorem~\ref{par-app} is not sufficient. However by similar ideas as in \cite{CoN} and as in \cite{De}, we can pass the distance distortion estimate to the limit space.  A difficulty here is that, as the discussion in \cite{De}, the distance distortion estimates for small balls centered in the interior of a limit unit speed geodesic $\gamma$ under the gradient flow of $\nabla d_{\gamma(0)}$ can not  be controlled obviously as in manifolds which is an important start point of an induction process. To overcome it,  we will estimate balls centered at points which can be chosen arbitrary close to the interior points of $\gamma$ and then take a limit. First we have the following two lemmas as in \cite[Lemma 5.1, Lemma 5.2]{De}. 
The gradient flow of $-\nabla d_{\gamma(0)}$, $\phi: X\times [0, \infty)\to X$, is defined as
$$\phi_t(x)=\left\{\begin{array}{cc} \gamma_{x, \gamma(0)}(t), & \text{ if } d(x, \gamma(0)\geq t;\\
\gamma(0), & \text{ if } d(x, \gamma(0))<t,\end{array}\right.$$
where $\gamma_{x, y}$ is a limit unit speed geodesic from $x$ to $y$.

\begin{Lem}[time extend] \label{time-extend}
Let the assumption be as in Theorem~\ref{sha-hol}, let $\nu$ be the Radon measure of $X$ defined as in Corollary~\ref{seg-lim} and let $l=1$.  Then there are $\epsilon_1(n, p, \sigma)>0, r_1(n, p,\sigma)>0, C(n, p)>1, c(n, p)>0$, such that for $r\leq r_1$, $\sigma\leq t_0\leq 1-\sigma$, there is $z\in B_r(\gamma(t_0))$ satisfying that for each $s\leq \epsilon_1$,

{\rm (4.3.1)} $C^{-1}\leq \frac{\nu(B_r(\phi_s(z)))}{\nu(B_r(z))}\leq C$, 

{\rm (4.3.2)} there is $A\subset B_r(z)$, $\nu(A)\geq (1- c)\nu(B_r(z))$ and $ \phi_s(A)\subset B_{2r}(\phi_s(z))$,

{\rm (4.3.3)} $e(z)\leq c(n, p,\sigma)r^2$.
\end{Lem}

\begin{Lem}[radius extend] \label{radius-extend}
Let the assumption be as in Lemma~\ref{time-extend}. Then there are $\epsilon_2(n, p,\sigma)>0, r_2(n, p, \sigma)>0$, such that for $\sigma\leq t_0\leq 1-\sigma$ and $s\leq \epsilon_2$, if (4.3.1)-(4.3.3) hold for some $z\in B_r(\gamma(t_0))$ and $r<r_2$, then (4.3.1)-(4.3.3) hold for $r'\in[4r, 16r]$ and the same $z$.
\end{Lem}

The proof of Lemma~\ref{time-extend} and Lemma~\ref{radius-extend} is similar as in \cite{De}. A difference is that we will do all estimates on manifolds and then pass to the limit. Here we give a rough proof of Lemma~\ref{time-extend}. 
\begin{proof}[Proof of Lemma~\ref{time-extend}]
Let $e(y)=d(y, \gamma(0))+d(y, \gamma(1))-1$. By Theorem~\ref{par-app}, for $r<\frac1{10}\sigma$,
$$-\kern-1em\int_{B_{r}(\gamma(t_0))}e(y)\leq C(n, p, \sigma) r^2.$$
Then by Corollary~\ref{spl-thm} and relative volume comparison \eqref{vol-comp}, for $r\leq r(n, p, \sigma)$, there is $z\in B_r(\gamma(t_0))$, $\epsilon=\epsilon(n,\sigma, r)$, such that for any $s\leq \epsilon$, (4.3.1)-(4.3.3) hold.

To remove the dependence of $r$ in $\epsilon$, we will show that there are $r_1(n, p, \sigma), \epsilon_1(n, p, \sigma)$ such that for $r<r_1$, for $z\in B_r(\gamma(t_0))$, if (4.3.1)-(4.3.3) hold for $s<\epsilon\leq \epsilon_1$, then there is $z'\in B_r(\gamma(t_0))$ satisfying (4.3.3) and  for $s<\epsilon$,
 
 {\rm (4.3.1')} $2C^{-1}\leq \frac{\nu(B_r(\phi_s(z')))}{\nu(B_r(z'))}\leq \frac{C}2$, 

{\rm (4.3.2')} there is $A'\subset B_r(z')$, $\nu(A')\geq (1- c)\nu(B_r(z'))$ and $ \phi_s(A')\subset B_{\frac{3r}{2}}(\phi_s(z'))$.

To find $z'\in B_r(\gamma(t_0))$, consider
$$dt_1(s)(u, v)=\min\{r, \max_{0\leq \tau \leq s}|d(u, v)-d(\phi_s(u),\phi_s(v))|\}.$$

Let $\phi_{s, i}$ be the gradient flow of $-\nabla d_{\gamma_i(0)}$ in $M_i$ such that $\phi_{s, i}\to \phi_s$ and let 
$$dt^i_1(s)(u_i, v_i)=\min\{r, \max_{0\leq \tau \leq s}|d(u_i, v_i)-d(\phi_{s, i}(u_i),\phi_{s, i}(v_i))|\}.$$ 
Then $dt_1^i(s)(u_i, v_i)\to dt_1(s)(u, v)$ if $u_i\to u, v_i\to v$. 

Assume $z_i\in B'_r(\gamma_i(t_0))=B_r(\gamma_i(t_0))\cap\{y, \, e_i(y)\leq c_1 r^2+ r^{2-\frac{n}{2p}}\delta_i\}$ and $z_i\to z$. Let $A_i\subset B_r(z_i)$, $A_i\to A$. Then for $s\leq \epsilon$, $\phi_{s, i}(A_i)\subset B_{2r+\Psi_i}(\phi_{s, i}(z_i))$, where $\Psi_i\to 0$ as $i\to \infty$. Let $D_1=A_i\cap B'_r(\gamma_i(t_0))$ and let
$$U_1^i(s)=\left\{(u_i, v_i)\in D_1\times B'_{2r}(z_i),\, dt^i_1(s)(u_i, v_i)<r\right\}.$$
Let $h_i^{\pm}=b_{t}^{\pm}$ be the parabolic approximation on $M_i$ as in Theorem~\ref{par-app}. Since
\begin{eqnarray*}
& & \frac{d}{ds}\int_{D_1\times B'_{2r}(z_i)} dt_1^i(s)(u, v)\\
&\leq & \int_{U_1^i(s)}\left(|\nabla h_i^{+}-\nabla b^{+}|(\phi_{s, i}(u))-|\nabla h_i^+-\nabla b^+|(\phi_{s, i}(v))\right) + \int_{U_1^i(s)}\left|\int_{c_s(u, v)}\op{Hess}_{h_i^+}(c_s(u, v)(\tau))\right|,
\end{eqnarray*}
where $c_s(u, v)$ is the minimal geodesic from $\phi_{s,i}(u)$ to $\phi_{s, i}(v)$,
and by segment inequality Theorem~\ref{seg-ine} and \eqref{vol-comp},
\begin{eqnarray*}
& &\int_{U_1^i(s)}\left|\int_{c_s(u, v)}\op{Hess}_{h_i^+}(c_s(u, v)(\tau))\right|\\
&\leq & \int_{(\phi_{s, i}, \phi_{s, i})(U_1^i(s))}\left|\int_{\gamma_{u, v}}\op{Hess}_{h_i^+}(\gamma_{u, v}(\tau))\right|\\
&\leq & \int_{B_{2r+\Psi_i}(\phi_{s, i}(z_i))\times B_{7r+\Psi_i}(\phi_{s, i}(z_i))}\left|\int_{\gamma_{u, v}}\op{Hess}_{h_i^+}(\gamma_{u, v}(\tau))\right|\\
&\leq & c(10r+\Psi_i)\volume(B_{7r+\Psi_i}(\phi_{s, i}(z_i)))^2\left(-\kern-1em\int_{B_{10r+\Psi_i}(\phi_{s, i}(z_i))}|\op{Hess}_{h_i^+}| + (10r+\Psi_i)^{-\frac{n}{2p}}\delta_i^{\frac12}\right)\\
&\leq & \sup_{s\in[0, \epsilon]}\frac{\volume(B_{7r+\Psi_i}(\phi_{s, i}(z_i)))^2}{\volume(B_{r}(z_i))^2}c(10r+\Psi_i)\volume(B_{r}(z_i))^2\left(-\kern-1em\int_{B_{10r+\Psi_i}(\phi_{s, i}(z_i))}|\op{Hess}_{h_i^+}| + (10r+\Psi_i)^{-\frac{n}{2p}}\delta_i^{\frac12}\right),
\end{eqnarray*}
\begin{eqnarray*}
& & \int_0^{\epsilon}\int_{U_1^i(s)}|\nabla h_i^{+}-\nabla b^{+}|(\phi_{s, i}(u)) \\
&\leq & \volume(B_{2r}(z_i))\int_0^{\epsilon}\int_{D_1}|\nabla h_i^{+}-\nabla b^{+}|(\phi_{s, i}(u))\\
&\leq & \volume(B_{2r}(z_i))^2\sqrt{\epsilon}\left(\int_0^{\epsilon}-\kern-1em\int_{D_1}|\nabla h_i^{+}-\nabla b^{+}|^2(\phi_{s,i}(u))\right)^{\frac12}\\
&\leq & \volume(B_{2r}(z_i))^2\sqrt{\epsilon}\left(\int_0^{\epsilon}-\kern-1em\int_{D_1}\left| |\nabla h_i^{+}|^2-1\right|(\phi_{s, i}(u))+ 2\left|\left<\nabla h_i^+,\nabla b^{+}\right>-1\right|(\phi_{s, i}(u))\right)^{\frac12}\\
& \leq & c\sup_{s\in[0, \epsilon]}\frac{\volume(B_{2r+\Psi_i}(\phi_{s, i}(z_i)))^{\frac12}}{\volume(B_{r}(z_i))^{\frac12}}\volume(B_r(z_i))^2\sqrt{\epsilon}\left(\int_0^{\epsilon}-\kern-1em\int_{B_{2r+\Psi}(\phi_{s, i}(z_i))}\left| |\nabla h_i^{+}|^2-1\right|(u)\right.\\
& & \left.+ 2\left|\left<\nabla h_i^+,\nabla b^{+}\right>-1\right|(u)\right)^{\frac12}\\
\end{eqnarray*}
for $t=(10r+\Psi)^2$, by Corollary~\ref{int-par} we have that
$$\int_{D_1\times B'_{2r}(z_i)} dt_1^i(s)(u, v)\leq \sup_{s\in[0, \epsilon]}\frac{\volume(B_{7r+\Psi_i}(\phi_{s, i}(z_i)))^2}{\volume(B_{r}(z_i))^2}C\sqrt{\epsilon}\volume(B_r(z_i))^2(r+ r^{1-\frac{n}{2p}}\delta_i+\Psi_i).$$

Passing to the limit we have that, for
 $B'_{2r}(z_i)\to B'_{2r}(z)=B_{2r}(z)\cap \{y,\, e(y)\leq c_1r^2\}$, $D_1=A_1\cap B'_r(\gamma(t_0))$, and for $s\leq \epsilon$
$$\int_{D_1\times B'_{2r}(z)} dt_1(s)(u, v)\leq C\sqrt{\epsilon}r\nu(B_r(z))^2.$$
Choose $z'\in D_1$ and let $B''_r(z')=B_r(z')\cap B'_{2r}(z)$. By integral excess estimate, we have  $$\frac{\nu(B''_r(z'))}{\nu(B_r(z'))}>c(n, p, \sigma)>0.$$
Thus
$$-\kern-1em\int_{B''_{r}(z')}dt_1(s)(z', v)\leq C'\sqrt{\epsilon} r .$$
When $\epsilon<\epsilon_1(n, p, \sigma)$, we have 
$$-\kern-1em\int_{B''_r(z')}dt_1(s)(z', v)\leq \frac14 r .$$
Thus for any $\eta>0$, there is $A'_{\eta}\subset B''_r(z')\subset B_r(z')$, such that $\frac{\nu(A'_{\eta})}{\nu(B''_r(z'))}>\eta$ and for any $v\in A'_{\eta}$,
$$dt_1(s)(z', v)\leq \eta^{-1}\frac14 r.$$

To prove (4.3.1'), first note that by (4.3.2'), for $s\in [0, \epsilon]$, $\nu(B_{3r/2}(\phi_s(z')))\geq \nu(\phi_s(A'))\geq (1+c(n)\epsilon)^{-n}\nu(A')$. Thus for $s<\epsilon<\epsilon_1(n, p, \sigma)$
$$\frac{\nu(B_r(\phi_s(z')))}{\nu(B_r(z'))}\geq \frac{\svolball{-1}{r}}{\svolball{-1}{\frac{3r}2}}\frac{\nu(B_{\frac{3r}2}(\phi_s(z')))}{\nu(B_r(z'))}\geq 2C^{-1}.$$
To prove the other side bound, as in \cite{De}, we will find a sufficient large subset of $B_{r}(\phi_s(z'))$ which stay close to $z'$ under a flow that does not decrease the measure significantly. Consider the flows of $-\nabla h_i^{+}$, $\tilde\phi_{s, i}$, and the flow of $\nabla h_i^{+}$, $\tilde \phi_{-s, i}$, on $M_i$ and then pass them to the limit to derive flow $\tilde \phi_{s}$ and $\tilde \phi_{-s}$ on $X$.

Let
$$dt_2(s)(u, v)=\min\{r, \max_{0\leq \tau \leq s}\{d(u, v)-d(\phi_{\tau}(u), \tilde \phi_{\tau}(v))\}\}.$$
For $A''=A'\cap \{y, \, e(y)\leq cr^2\}$, as the above discussion, we can show that
$$\int_{A''\times B_r(z')}dt_2(s)\leq c\sqrt{\epsilon}r\nu(B_r(z))^2.$$
Choose $z_1\in A''$, then there is $D_2\subset B_r(z')$ and $\epsilon\leq \epsilon_1(n, p, \sigma)$ such that for $s\leq \epsilon$,
$$\frac{\nu(D_2)}{\nu(B_r(z'))}\geq 1-c, \quad \forall\, v\in D_2,  dt_2(s)(z_1, v)\leq \frac{r}{2},$$
and thus 
$$\tilde \phi_s(D_2)\subset B_{4r}(\phi_s(z')), \quad \frac{\nu(\tilde \phi_s(D_2))}{\nu(B_r(z))}\geq C(n, p, \sigma).$$

Now flow $\tilde \phi_s(D_2)$ back by $\tilde \phi_{-s}$. Consider
$$dt_3(s)(u, v)=\min\{r, \max_{0\leq \tau\leq s}\{d(u, v)-d(\tilde\phi_{-\tau}(u), \tilde \phi_{-\tau}(v))\}\},$$
we have
$$\int_{\tilde\phi_s(D_2) \times B_r(\phi_s(z'))}dt_3(s)\leq c\sqrt{\epsilon}r\nu(B_r(z))^2.$$
Choose $z_2\in \tilde\phi_s(D_2)$, then there is $D_3\subset B_r(\phi_s(z'))$ such that
$$\frac{\nu(D_3)}{\nu(B_r(\phi_s(z')))}\geq 1-c, \quad \forall\, v\in D_3, dt_3(s)(z_2, v)\leq \frac12r,$$
and thus
$$\tilde \phi_{-s}(D_3)\subset B_{7r}(z').$$
Then for $\epsilon\leq \epsilon_1(n, p, \sigma)$.
$$\frac{\nu(B_r(\phi_s(z')))}{\nu(B_r(z'))}\leq \frac{\svolball{-1}{7r}}{\svolball{-1}{r}}\frac{\nu(B_r(\phi_s(z')))}{\nu(B_{7r}(z'))}\leq \frac{\svolball{-1}{7r}}{\svolball{-1}{r}}\frac{1}{1-c}\frac{D_3}{\nu(B_{7r}(z'))}\leq \frac{C}{2}.$$
\end{proof}

By above two lemmas, as \cite[Lemma 5.3]{De}, we have that
\begin{Cor} \label{main}
Let the assumption be as in Lemma~\ref{time-extend}. Then there are $\epsilon_3(n, p, \sigma)>0, r_3(n, p, \sigma)>0$, such that for $\sigma\leq t_0\leq 1-\sigma$ and $r\leq r_3$, there is $z\in B_r(\gamma(t_0))$ satisfying (4.3.1)-(4.3.3) for any $r\leq r'\leq r_3$ and $s\leq \epsilon_3$.
\end{Cor}

In Corollary~\ref{main}, take $r_i\to 0$, $z_i\in B_{r_i}(\gamma(t_0))$, we have that (4.3.1), (4.3.2) hold for $r\leq r_3$ and $\gamma(t_0)$. Unlike \cite{De}, we can choose $z_i$ and limit geodesics $\gamma_{\gamma(0), z_i}$ and thus $\phi_s(z_i)$ such that $\gamma_{\gamma(0), z_i}\to \left.\gamma\right|_{[0, t_0]}$. Consider the gradient flow $-\nabla d_{\gamma(1)}$, $\psi: X\times [0, \infty]\to X$, 
$$\psi_t(x)=\left\{\begin{array}{cc} \gamma_{x, \gamma(1)}(t), & \text{ if } d(x, \gamma(1)\geq t;\\
\gamma(1), & \text{ if } d(x, \gamma(1))<t.\end{array}\right.$$
Same discussion gives (as \cite[Theorem 5.9]{De})
\begin{Lem}
Let the assumption be as in Lemma~\ref{time-extend}. Then there are $\epsilon_0(n, p, \sigma)>0, r_0(n, p, \sigma)>0, C(n, p)>1, c(n, p)>0$, such that for $r\leq r_0$, $s\leq t_1-t_0\leq \epsilon_0$,

{\rm (4.6.1)} $C^{-1}\leq \frac{\nu(B_r(\gamma(t_1)))}{\nu(B_r(\gamma_0))}\leq C$;

{\rm (4.6.2)} There is $A\subset B_r(t_1)$, $\nu(A)\geq (1-c)\nu(B_r(\gamma(t_1)))$, $\phi_s(A)\subset B_{2r}(\gamma(t_1-s))$;

{\rm (4.6.3)} There is $B\subset B_r(t_0)$, $\nu(B)\geq (1-c)\nu(B_r(\gamma(t_0)))$, $\psi_s(B)\subset B_{2r}(\gamma(t_0+s))$.
\end{Lem}

Now Theorem~\ref{sha-hol} derives as in \cite{CoN} (see also \cite{De}). In fact, we can estimate more carefully  in the proof of Lemma~\ref{time-extend} to derive that (4.6.2'): for small $\eta=\eta(n, p, \sigma)>0$, there is $A_{\eta}\subset B_{r}(\gamma(t_1))$ such that  $\nu(A)\geq (1-\Psi(\eta))\nu(B_r(\gamma(t_1)))$, and for $s<\Psi(\eta)$, $\phi_s(A_{\eta})\subset B_{r+\Psi(\eta)}(\gamma(t_0))$ and for each $u, v\in A_{\eta}$, $$|d(u, v)-d(\phi_s(u), \phi_s(v))|\leq \Psi(\eta)r.$$ 

In the following, we will give some applications of Theorem~\ref{sha-hol}.
Assume that $\left(M_i, x_i, g_i, \frac{\volume(\cdot)}{\volume(B_1(x_i))}\right)\overset{GH}\to (X, x, d, \nu)$, where $(M_i, x_i)$ is a sequence of complete $n$-manifolds with $\bar k_i(H, p, 1)\leq \delta_i\to 0$.
For each $x\in X$, a tangent cone $(T_x, x^*)$ at $x$ is a Gromov-Hausdorff limit of $(r_i^{-1}X, x)=(X, r_i^{-1}d, x)$, where $r_i\to 0$. Passing to a subsequence, we have that $(r_i^{-1}M_i, x_i)\to (T_x, x^*)$ and by \eqref{small}, $\bar k_{r_i^{-1}M_i}(0, p, 1)\to 0$. 
Then the splitting theorem Corollary~\ref{spl-thm} holds on $T_x$ and we may assume that there exist some integer $k\geq 0$ and a length space $Y$ such that 
$$(T_x, x^*)=(\Bbb R^k\times Y, (0, y^*)).$$
 As in \cite{CC2}, let $\mathcal R_k=\{x\in X, \, \text{ each }T_x=\Bbb R^k \}$, let $\mathcal R=\cup_{k=1}^n \mathcal R_k$ be the regular set and let $\mathcal S=X\setminus \mathcal R$ be the singular set. 
We call a set $A\subset X$ is a.e. convex if for $\nu\times \nu$ a.e. point $(y, z)\in A\times A$ there exists a minimal geodesic $\gamma\subset A$ between $x$ and $y$.
And $A$ is weakly convex if for any $y, z\in A$, for each $\epsilon>0$, there is a curve $\gamma\subset A$ which connects $y, z$ and $\left||\gamma|-d(y, z)\right|\leq \epsilon^2d(y, z)$. 
\begin{Thm}[Regularity of the limit spaces] \label{reg-str}
Let $(X, x, d, \nu)$ be as above. Then

{\rm (4.7.1)} $\nu(\mathcal S)=0$;

{\rm (4.7.2)} There is a unique $k$, such that $\nu(\mathcal R\setminus \mathcal R_k)=0$;

{\rm (4.7.3)} The full measure set $\mathcal R_k$ is a.e. convex and weakly convex.
\end{Thm}

The proof of (4.7.1) is just the same as in \cite[Section 2]{CC2} by using Theorem~\ref{spl-map} and some basic geometric measure theory under local doubling property (relative volume comparison \eqref{vol-comp}). And (4.7.3) follows by using (4.7.1), (4.7.2) and a similar argument as in \cite{CoN}. To prove (4.7.2), in \cite{CoN} Colding-Naber used  the following two results and the sharp H\"older continuity in the limit space Theorem~\ref{sha-hol}:

(a) For two open subset $A_1, A_2\in X$ and a measurable subset $A\subset X$, if $\nu(A)=0$, then for $\nu\times \nu$ almost every $(a_1, a_2)\in A_1\times A_2$
$$\inf\left|\{t: \gamma_{a_1 a_2}(t)\in A\}\right|=0,$$
where the infimum is taking over all minimal limit geodesics $\gamma_{a_1a_2}$ connecting $a_1$ and $a_2$;

(b) a.e. $(y, z)\in X\times X$ is in the interior of a limit geodesic.

(a) is a direct application of the segment inequality in the limit space (Corollary~\ref{seg-lim}) to the indicator function of $A$.  

To prove (b), for each $r>0, N>0$ an integer, consider
$$\op{Cl}(X, r, N^{-1})=\{(x, y)\in X\times X, \, e_{z, w}(x, y)\geq N^{-2}, \forall\, (z, w)\notin B_{\sqrt 2 r}(x, y)\},$$
where $$e_{z, w}(x, y)=\frac{1}{\sqrt 2}d(x, y)+d_{X\times X}((x, y), (z, w))-\frac{1}{\sqrt 2}d(z, w).$$
It is obvious that if $(x, y)\notin \cup_{N}\op{Cl}(X, r, N^{-1})$ then there is a minimal geodesic $\gamma: [0, l]\to X$ such that there exist $s, t\in [r, l-r]$, $\gamma(s)=x$ and $\gamma(t)=y$, i.e., there is a minimal  geodesic between $x$ and $y$ that can be extended $r$ longer.
Thus to prove (b) it sufficient to show that for any $\sigma>0$, $x\in X$
$$\nu\left(\cap _{r}\left(\cup_{N}\op{Cl}(X, r, N^{-1})\cap A_{\sigma, \sigma^{-1}}(S)\right)\right)=0,$$
or
\begin{equation}\nu\left(\cup_{N}\op{Cl}(X, r, N^{-1})\cap A_{\sigma, \sigma^{-1}}(S)\right)\leq Cr, \label{cl-est}\end{equation}
where $S=\{(w, w), \, w\in B_1(x)\}$ and $A_{\sigma, \sigma^{-1}}(S)=\{(y, z)\in X\times X, \, \sqrt 2 \sigma\leq d(y, z)\leq \sqrt 2\sigma^{-1}, \op{Proj}_D(y, z)\in S\}$, $\op{Proj}_D: X\times X \to \Delta=\{(x, x),\, x\in X\}$ is the projection map. 

Note that the distance to $\Delta$  can be written as $d_D(x, y)=\frac{1}{\sqrt 2}d(x, y)$. 

To get \eqref{cl-est}, we will first estimate the set $\op{Cl}(M_i, r)=\cup_{N}\op{Cl}(M_i, r, N^{-1})$ in the sequence of manifolds and then use its stability under the Gromov-Hausdorff convergence as in \cite{CoN}.

By \eqref{vol-element}, we have
\begin{eqnarray*}
& &\volume(\op{Cl}(M, r)\cap A_{\sigma, \sigma^{-1}}(S)) \\
&=& \left(\sqrt 2\right)^{2n}\int_{B_1(x)}\int_{\sqrt 2\sigma}^{\sqrt 2\sigma^{-1}}\int_{S^{n-1}}\chi_{\op{Proj}_1(\op{Cl}(M, r))}\mathcal A(t, \theta) d\theta dt d\volume(z)\\
&\leq & \left(\sqrt 2\right)^{2n}r\int_{B_1(x)}\int_{S^{n-1}} \frac{\underline{\mathcal A}_H(\sqrt 2\sigma^{-1})}{\underline{\mathcal A}_H(\sqrt 2\sigma)} \left(\mathcal A(\sqrt 2\sigma, \theta)+\int_0^{\sqrt 2\sigma^{-1}}\psi(\tau, \theta)\mathcal A(\tau, \theta)d\tau\right) d\theta d\volume(z)\\
&\leq & c(n, H, \sigma)r \left(\int_{B_1(x)}\volume(\partial B_{\sqrt 2\sigma}(z)) + \volume(B_{\sqrt 2\sigma^{-1}}(z))\bar k^{\frac12}(H, p, \sqrt 2\sigma^{-1})\right)d\volume(z)\\
&\overset{\eqref{sph-ball}}\leq & c(n, H, \sigma)r \left(\int_{B_1(x)}\left(\frac{\svolsp{H}{\sqrt 2\sigma}}{\svolball{H}{\sqrt 2\sigma}} + c(n, p)\bar k^{\frac12}(H, p, \sqrt 2\sigma)\right)\volume(B_{\sqrt 2\sigma}(z)) + \volume(B_{\sqrt 2\sigma^{-1}}(z))\bar k^{\frac12}(H, p, \sqrt 2\sigma^{-1})\right)\\
& \overset{\eqref{vol-com}}\leq & C(n, H, \sigma) r (1+\bar k^{\frac12}(H, p, 1)),
\end{eqnarray*}
where we used \eqref{vol-com} and \eqref{sph-ball}, $\bar k(H, p, r) \leq c(n, p, H)\frac{\svolball{H}{1}}{\svolball{H}{r}}\bar k(H, p, 1)$ for $r<1$ and Lemma~\ref{pack}.

By Theorem~\ref{reg-str}, as the discussion in \cite{CC3, CoN}, we have that
\begin{Thm}[Lie isometry group] \label{lie-gro}
The isometry group of a limit space $X$ as in Theorem~\ref{reg-str} is a Lie group.
\end{Thm}

As a simple application of Theorem~\ref{sha-hol} and Theorem~\ref{reg-str}, as in \cite{Che}, we have that
\begin{Cor} \label{one-dim}
Let $(X, x)$ be as in Theorem~\ref{reg-str}. If  $\mathcal R_1\neq \emptyset$, then $X$ is a one dimensional topological manifold which may have boundary. 
\end{Cor}

\end{document}